\documentclass{amsart}
\usepackage{amsmath}
\usepackage{amsfonts}
\usepackage{graphics}
\usepackage{epsfig}
\usepackage{amssymb}
\usepackage{amscd}
\usepackage[all]{xy}
\usepackage{latexsym}
\usepackage{graphicx}
\usepackage{multirow}
\usepackage{geometry}

\newtheorem{thm}{Theorem}[section]

\newtheorem{lem}[thm]{Lemma}
\newtheorem{prop}[thm]{Proposition}

\theoremstyle{definition}

\newtheorem{rem}[thm]{Remark}
\newtheorem*{ack}{Acknowledgement}

\numberwithin{equation}{section}
\numberwithin{figure}{section}

\def\Hom{{\text{\rm{Hom}}}}

\def\rchi{{\hbox{\raise1.5pt\hbox{$\chi$}}}}

\def\const{{\text{\rm{const}}}}

\def\a{\alpha}
\def\b{\beta}
\def\lam{\lambda}

\def\gam{\gamma}

\newcommand{\bea}{\begin{eqnarray}}
\newcommand{\eea}{\end{eqnarray}}
\newcommand{\be}{\begin{equation}}
\newcommand{\ee}{\end{equation}}

\newcommand{\Mbar}{{\overline{\mathcal{M}}}}

\newcommand{\bP}{{\mathbb{P}}}
\newcommand{\bC}{{\mathbb{C}}}

\newcommand{\bZ}{{\mathbb{Z}}}

\newcommand{\cD}{{\mathcal{D}}}

\newcommand{\half}{{\frac{1}{2}}}
\newcommand{\bp}{{\mathbf{p}}}

\newcommand{\rar}{\rightarrow}
\newcommand{\lrar}{\longrightarrow}

\begin{document}
\large
\setcounter{section}{0}

\title[Quantum curves for Hurwitz numbers of 
 arbitrary base]
{Quantum curves for simple 
Hurwitz numbers of an arbitrary base curve}

\author[X.\ Liu]{Xiaojun Liu}
\address{
Department of Applied Mathematics\\
China Agricultural University\\
Beijing, 100083, China\\
and
Department of Mathematics\\
University of California\\
Davis, CA 95616--8633, U.S.A.}
\email{xjliu@cau.edu.cn}

\author[M.\ Mulase]{Motohico Mulase}
\address{
Department of Mathematics\\
University of California\\
Davis, CA 95616--8633, U.S.A.}
\email{mulase@math.ucdavis.edu}

\author[A.\ Sorkin]{Adam  Sorkin}
\address{
Department of Mathematics\\
University of California\\
Davis, CA 95616--8633, U.S.A.}
\email{azsorkin@math.ucdavis.edu}

\begin{abstract}
The generating functions of simple Hurwitz numbers
of the projective line are known to satisfy many
properties. They include a heat equation,
the Eynard-Orantin topological
recursion, an infinite-order differential equation
called a quantum curve equation, and a
Schr\"odinger like partial differential equation. 
In this paper we generalize these
properties to simple Hurwitz numbers with 
an arbitrary base curve. 
\end{abstract}


\maketitle

\allowdisplaybreaks

\tableofcontents

\section{Introduction and the main results}
\label{sect:intro}

The purpose of this paper is to determine
functional properties of various generating
functions of simple Hurwitz numbers with 
an arbitrary fixed base curve $B$. We derive 
a partial differential  equation 
for the Laplace transform of these 
Hurwitz numbers. The equation
is completely analogousΠto the result of
\cite{MZ} for the usual
simple Hurwitz numbers based on
$\bP^1$. 
We  also obtain an infinite-order 
 differential equation, or  
   a \emph{quantum curve},
    for the case of an arbitrary
base curve,
that is parallel to the various Hurwitz problems
studied in 
\cite{BHLM, MSS,MS} with the base curve
$\bP^1$ and its twisted version $\bP^1[a]$.

Our main motivation is to 
examine whether the 
 topological recursion of Eynard-Orantin
\cite{EO1,EO3} and the 
existence of a quantum curve of
\cite{ADKMV,DHS,DHSV,DV2007,GS} hold for
the enumeration problem 
of simple Hurwitz numbers over 
an arbitrary curve $B$. Consider a holomorphic 
map $f:C\lrar B$ of a non-singular algebraic 
curve $C$ onto a fixed base curve $B$. 
We choose a general point $0\in B$
and fix it once for all.
The quantity 
we are interested in this paper is 
the base $B$ Hurwitz number
$
H_{g,n} ^B (\mu_1,\dots,
\mu_n),
$
which counts the automorphism weighted 
number of the topological types of holomorphic
maps $f$ of a genus $g$ domain curve with 
$n$ labeled preimages of $0\in B$ of
multiplicity $ (\mu_1,\dots,
\mu_n)\in \bZ_+^n$, such that $f$ is simply 
ramified other than these $n$ points. Define its
discrete Laplace transform
\begin{equation}
\begin{aligned}
\label{eq:Fgn}
F_{g,n} ^B(x_1,\dots,x_n) &=
\sum_{\vec{\mu}\in\bZ_+^n}H_{g,n} ^B(
\mu_1,\dots,\mu_n)
\prod_{i=1}^n e^{-w_i\mu_i}
\\
&=
\sum_{\vec{\mu}\in\bZ_+^n}H_{g,n} ^B(
\mu_1,\dots,\mu_n)
\prod_{i=1}^n x_i ^{\mu_i}, \quad x_i=e^{-w_i},
\end{aligned}
\end{equation}
which we call the \emph{free energy} of type $(g,n)$.
Since $\bP^1$ does not cover the base curve $B$
of genus $g(B)>0$,
$$
F_{0,n}^B(\mu_1,\dots,\mu_n) =0
$$
for any $n\ge 1$ and any value of $(\mu_1,\dots,\mu_n)
\in \bZ_+^n$ in this case. 
Our main result is the following.

\begin{thm}
\label{thm:main}
For $2g-2+n  >0$, 
the free energies $F_{g,n} ^B(x_1,\dots,x_n)$
satisfy the following 
partial differential  equation:
\begin{multline}
\label{eq:PDE}
\left(
2g-2+n-\big(1-\rchi(B)\big)
\sum_{i=1}^n x_i\frac{\partial}{\partial x_i}
\right) F_{g,n} ^B\big(x_{[n]}\big)
\\
=
\half \sum_{i\ne j} \frac{x_ix_j}{x_i-x_j}
\left(
\frac{\partial}{\partial x_i}
F_{g,n-1} ^B\big(x_{[\hat{j}]}\big)
-
 \frac{\partial}{\partial x_j}
F_{g,n-1} ^B\big(x_{[\hat{i}]}\big)
\right)
\\
+
\half \sum_{i=1}^n u_1\frac{\partial}{\partial u_1}
u_2\frac{\partial}{\partial u_2}
\Bigg|_{u_1=u_2=x_i}
\\
\left[
F_{g-1,n+1}^B\big(u_1,u_2,x_{[\hat{i}]})
+
\sum_{\substack{g_1+g_2=g\\
I\sqcup J=[\hat{i}]}}
F_{g_1,|I|+1}^B(u_1,x_I) F_{g_2,|J|+1}^B(u_2,x_J)
\right].
\end{multline}
Here $[n]=\{1,\dots,n\}$ is an index set, 
$[\hat{i}] = [n]\setminus \{i\}$, and 
for any subset $I\subset [n]$, 
we denote $x_I = (x_i)_{i\in I}$. 
We denote by $\rchi(B) = 2-2g(B)$ the
Euler characteristic of the base curve $B$. 
Note that the complexity $2g-2+n$ is 
reduced by $1$ on the right-hand side of the
recursion when $g(B)\ge 1$, 
similar to the Eynard-Orantin 
integral recursion of \cite{EO1,EO3}.

Let us define the \textbf{partition function} 
of the base $B$ Hurwitz numbers, as 
a holomorphic function in $x\in \bC$ and 
$\hbar$ with $Re(\hbar)<0$, by
\begin{equation}
\label{eq:ZB}
Z^B(x,\hbar) = \exp\left(\sum_{g=1}^\infty 
\sum_{n=1}^\infty
\frac{1}{n!}
\hbar^{2g-2+n} F_{g,n}^B(x,x,\dots,x)\right).
\end{equation}
Then it 
 satisfies a \textbf{quantum curve} like
 infinite-order differential 
equation
\begin{equation}
\label{eq:P}
\hbar x \frac{d}{dx}
\left[1- 
\hbar^{1-\rchi(B)}x\;  e^{\hbar x \frac{d}{dx}}
  \left(\frac{d}{dx} x
\right)^{1-\rchi(B)}
\right]
Z^B(x,\hbar)
=0.
\end{equation}
If we introduce 
\begin{equation*}
y = \hbar x \frac{d}{dx}
\end{equation*}
and regard it as a commuting variable, then the
total symbol of the above operator produces
a \textbf{Lambert curve} 
\begin{equation}
\label{eq:B-Lambert}
x = y^{\rchi(B)-1} e^{-y}.
\end{equation}
The partition function  also satisfies a
\textbf{Schr\"odinger} like partial differential equation
\begin{equation}
\label{eq:Q}
\left[
\frac{\partial}{\partial \hbar}-\half
\left(x \frac{\partial}{\partial x}\right)^2
+\left(\half-\frac{1-\rchi(B)}{\hbar}\right)
x \frac{\partial}{\partial x}
\right] Z^B(x,\hbar)
=0.
\end{equation}
If we denote the above operators by $P$ and $Q$, 
respectively,
then they satisfy a commutation relation:
\begin{equation}
\label{eq:[P,Q]}
[P,Q]=- \frac{1}{\hbar}\;P.
\end{equation}
This shows that the system of equations
{\rm{(\ref{eq:P})}} and 
{\rm{(\ref{eq:Q})}} are 
compatible. Moreover, 
the partition function has the
following simple expression
\begin{equation}
\label{eq:ZB sum}
Z^B(x,\hbar) =\sum_{m=0}^\infty
(m!)^{1-\rchi(B)}e^{\half m (m-1)\hbar} 
\left(\hbar^{1-\rchi(B)} x\right)^m.
\end{equation}
Although the parameter $\hbar$ is a formal
deformation parameter, if we write 
$\hbar = 2\pi i \tau$, then 
\begin{equation}
\label{eq:ZB in tau}
Z^B(x,2\pi i\tau) =\sum_{m=0}^\infty
(m!)^{1-\rchi(B)}e^{\pi i m (m-1)\tau} 
\left((2\pi i\tau)^{1-\rchi(B)} x\right)^m
\end{equation}
is an entire function in $x$ for $Im(\tau) >0$.
\end{thm}

\begin{rem}
\label{rem:orbifold}
We  note that the exact same 
formulas (\ref{eq:PDE}),
(\ref{eq:B-Lambert}), (\ref{eq:Q}),
and (\ref{eq:ZB sum}) hold for the case of 
an orbifold base $B=\bP^1[a]$, $a>0$, if we 
evaluate $\rchi(B) = 1+\frac{1}{a}$, 
interpret the sum in (\ref{eq:ZB sum}) as
 running over non-negative multiples
$am$ of
$a$, and replace $(am)!$ by 
$
(am)! \longmapsto (m! a^m)^a.
$
Since the orbifold case requires a different set of 
preparations (see \cite{BHLM,DLN,MSS}), 
we will report the generalization to the
case of higher-genus twisted curve as a
base elsewhere.
\end{rem}

When the base curve $B$ is an elliptic curve 
$E$, a special case of simple Hurwitz numbers 
exhibits a \emph{quasi-modular} property.
Dijkgraaf \cite{D} considered a 
generating function
\begin{equation}
\label{eq:Fg}
F_g(q) = \sum_{n=1} ^\infty \frac{1}{n!}
H_{g,n}^E(1,1,\dots,1) q^n
\end{equation}
for $g>1$,
and 
\begin{equation}
\label{eq:F1}
F_1(q) = -\frac{1}{24} \log q + 
\sum_{n=1} ^\infty \frac{1}{n!}
H_{1,n}^E(1,1,\dots,1) q^n.
\end{equation}
The significant fact here is that 
$F_g(q)$ for $g>1$ is a quasi-modular form
 \cite{D,KZ}.
The relation between 
the free energies of type $(g,n)$ for all $n$ and 
$F_g(q)$ is given by
\begin{equation}
\label{eq:Fg and Fgn}
F_g(q)  = \lim_{\lam \rar 0}\sum_{n=1}^\infty 
\frac{1}{n!} \frac{1}{\lam^n}
F_{g,n}^E (\lam q,\lam q,\dots,\lam q).
\end{equation}

To motivate  the present work, let us
recall
the corresponding counting problem of
simple Hurwitz number
$
H_{g,n} ^{\bP^1} (\mu_1,\dots,
\mu_n)
$
for
a pointed projective line $(\bP^1,\infty)$, which
has a long history since Hurwitz \cite{H}.
Its modern interest is due, among other 
things, to its rich functional
properties  \cite{Goulden,GJ,Kazarian,V, Zhou1},
its relation with linear Hodge integrals on
$\Mbar_{g,n}$ \cite{ELSV, GV1,OP1},
and the Bouchard-Mari\~no conjecture
\cite{BM} and its solutions \cite{BEMS,EMS,MZ}.
The theorem established in \cite{EMS,MZ} shows
that the discrete Laplace transform
\begin{equation}
\label{eq:FgnP}
F_{g,n}^{\bP^1}(x_1,\dots,x_n)
=
\sum_{\vec{\mu}\in\bZ_+^n}H_{g,n} ^{\bP^1}(
\mu_1,\dots,\mu_n)
\prod_{i=1}^n x_i ^{\mu_i}
\end{equation}
satisfies an integral recursion formula that was
originally
proposed by Eynard and Orantin \cite{EO1}, 
as conjectured in \cite{BM}.
The input curve for the recursion, the \emph{spectral
curve}, is shown to be the original Lambert curve
\begin{equation}
\label{eq:Lambert}
x = y e^{-y}.
\end{equation}
Then in \cite{MS, Zhou4} it is discovered that
there is a quantum curve
of \cite{ADKMV,DHS,DHSV, DV2007,GS},
 which is a 
generator of the holonomic system that characterizes
the partition function of the simple Hurwitz numbers
of $\bP^1$. The partition function for $B=\bP^1$
is
\begin{equation*}
Z^{\bP^1}(x,\hbar) = \exp\left(\sum_{g=0}^\infty 
\sum_{n=1}^\infty
\frac{1}{n!}
\hbar^{2g-2+n} F_{g,n}^{\bP^1}(x,x,\dots,x)\right).
\end{equation*}
It is established in \cite{MS} that the $\bP^1$
partition function has 
an expression 
\begin{equation}
\label{eq:ZP1 sum}
Z^{\bP^1} (x,\hbar) = \sum_{m=0}^\infty 
\frac{1}{m!} e^{\half m(m-1)\hbar}
\left(\frac{x}{\hbar}\right)^m,
\end{equation}
and that it satisfies two 
equations:
\begin{align}
\label{eq:PP1}
&\left[
\hbar x \frac{\partial}{\partial x}
-x e^{\hbar x\frac{\partial}{\partial x}}
\right] Z^{\bP^1}(x,\hbar) = 0,
\\
\label{eq:QP1}
&\left[
\frac{\partial}{\partial \hbar}
-\half \left(x\frac{\partial}{\partial x}\right)^2
+\left(\half +\frac{1}{\hbar}\right)
x\frac{\partial}{\partial x}
\right]Z^{\bP^1}(x,\hbar) = 0.
\end{align}
If we denote the above operators as $P_1$ and $Q_1$,
respectively, then they satisfy a commutation relation
\begin{equation}
\label{eq:P1Q1}
[P_1,Q_1]=-\frac{1}{\hbar}P_1. 
\end{equation}
The semi-classical limit of each of the
equations (\ref{eq:PP1}) and (\ref{eq:QP1})
recovers the Lambert curve (\ref{eq:Lambert}),
as shown in \cite{MS}.
Then from this curve as the spectral curve,
and using the functions $x$ and $y$  on it,
the Eynard-Orantin integral recursion \cite{EO1}
determines the differential forms
\begin{equation}
\label{eq:diff}
d_1d_2\cdots d_n F_{g,n}^{\bP^1}
\end{equation}
defined on
$\Sigma^n$
for all $(g,n)$,
where $\Sigma$ is the Lambert curve given by
(\ref{eq:Lambert}). It is a simple consequence 
of the results in \cite{EMS,MS,MZ} that
the differential form (\ref{eq:diff}) uniquely 
recovers the primitive  $F_{g,n}^{\bP^1}$
 as a function in $(y_1,\dots,y_n)\in \Sigma^n$.
From this point of view, the generating function
$F_{g,n}^{\bP^1}\big(x(y_1),\dots,x(y_n)\big)$ is
completely determined by the equation
(\ref{eq:PP1}) or (\ref{eq:QP1}), where
$x$ as a function in $y$ is also given by 
(\ref{eq:Lambert}).

Our motivating question is, \emph{do the 
similar properties hold 
when we consider the simple Hurwitz numbers
with an arbitrary curve as the base?} Since the genus 
$0$ base $B$ Hurwitz numbers 
$H_{0,n}^E(\mu_1,\dots,\mu_n)$
do not exist for a base curve $B$
with $g(B)>0$, the philosophy of \cite{DMSS}
does 
not produce any spectral curve of the
Eynard-Orantin integral recursion for this 
counting problem. Yet we have a topological
recursion (\ref{eq:PDE}) in the form of a
partial differential equation.
We note the straightforward generalizations 
(\ref{eq:P}),  (\ref{eq:Q}), 
  (\ref{eq:ZB sum}), and (\ref{eq:B-Lambert}),
  that reduce
to  the $\bP^1$ case
 (\ref{eq:PP1}), (\ref{eq:QP1}), 
(\ref{eq:ZP1 sum}), and (\ref{eq:Lambert}),
respectively,
when $\rchi(B)=2$.

The paper is organize as follows. In
Section~\ref{sect:CAJ}, we derive the 
cut-and-join equation 
of simple Hurwitz numbers with an arbitrary
base curve.
Then in Section~\ref{sect:LT}, we 
consider the Laplace transform of the cut-and-join
equation, which is exactly
(\ref{eq:PDE}). Using this result we
 derive
 the Schr\"odinger equation 
 (\ref{eq:Q})
 in Section~\ref{sect:Sch}.
 Section~\ref{sect:heat} is devoted to 
 explaining the heat equation expression of
 the cut-and-join equation 
 \cite{Goulden, Kazarian,Zhou1}
 and
 deriving its consequences. 
 In Section~\ref{sect:quantum}
 we first prove
 the expansion (\ref{eq:ZB sum}), 
 derive the quantum curve (\ref{eq:P}),
 and then establish the commutator
 relation (\ref{eq:[P,Q]}).
 In the final section  we deduce the
 Lambert curve (\ref{eq:B-Lambert})
 as the semi-classical limit.

\section{A cut-and-join equation for 
simple Hurwitz numbers}
\label{sect:CAJ}

The enumeration problem we consider in this
paper is the number of topological types of
holomorphic maps $f:C\lrar B$ from a nonsingular
curve $C$ of genus $g$ to a fixed base
 curve $B$,
with an arbitrary ramification over one 
point on $B$
and simple ramification at all other critical points. 
Let us denote by $0\in B$  a \emph{general}
point arbitrarily 
chosen and fixed. 
Such a homomorphic map $f$ is referred to as
a \emph{simple Hurwitz cover} of $B$.
We label each inverse image of $0\in B$
via $f$, and  denote
 by $(\mu_1,\dots,\mu_n)\in\bZ_+^n$ the 
degrees  of $f$ at each inverse image of $0$.
The base $B$ \emph{Hurwitz number}
$H_{g,n}^B(\mu_1,\dots.\mu_n)$ counts the
automorphism weighted number of the topological
types of
simple Hurwitz covers.

The degree of the map $f$ is given by
\begin{equation}
\label{eq:d}
d = |\mu| = \sum_{i=1}^n \mu_i.
\end{equation}
Here use 
the notation $|\mu|$, 
borrowing   from 
the convention in the theory symmetric functions.
The Riemann-Hurwitz formula tells us that there
are 
\begin{equation}
\label{eq:r}
r = r(g,\mu) = 2g-2+n -d\big(2g(B)-1\big)
\end{equation}
simple ramification points.
Instead of counting the 
\emph{topological types} of $f$, we can fix 
$r$ simple branch points on $B$ in 
general position other than $0\in B$
and count the automorphism weighted 
holomorphic maps $f$.
We note that (\ref{eq:d}) and (\ref{eq:r})
imply a condition for the degree and the genus
for a base curve with $g(B)\ge 1$:
\begin{align}
\label{eq:nd}
&n\big(2g(B)-1\big)
\le d \big(2g(B)-1\big)
\le 2g-2+n,
\\
&\label{eq:g and g(B)}
n \big(g(B)-1\big)\le g-1.
\end{align}
In particular, $H_{g,n}(\vec{\mu}) = 0$
for any $n\ge 1$ if $g< g(B)$. This poses a sharp 
contrast to the Gromov-Witten theory 
of  curves \cite{OP2}.

The map $f:C\lrar B$ of degree $d$ is determined 
by a monodromy representation  
$$
\rho\in
\Hom\big(\pi_i(B\setminus \{0,1,\dots,r,\}),S_d\big)
$$
of the fundamental group of an $(r+1)$-punctured
 curve into the symmetric group of $d$ letters.
Here $\{1,\dots,r\}\subset B$ is the label
of $r$ simple branched points on $B$ chosen
in general position. 
Let $\gam_j$, $j=0, \dots,r$, be a simple
loop around each point $j\in B$. Then a
simple Hurwitz cover is constructed by assigning
each $\gam_j$ for $j\ge 1$ to a transposition 
and $\gam_0$ to a product of $n$ disjoint cycles
of length $\mu_1,\dotsm\mu_n$, subject 
to the commutator relation
$$
\rho(\gam_1)\cdots\rho(\gam_r)\rho(\gam_0)
= [\a_1,\b_1]\cdots [\a_{g(B)},\b_{g(B)}],
$$
where $\a_{i},\b_i\in S_d$ are some elements in the
permutation group.

The  \emph{cut-and-join equation}
\cite{Goulden,GJ,H, V} is the result of 
an analysis of what happens when 
we multiply a transposition  
$\rho(\gam_r)$
to a product of disjoint cycles 
$\rho(\gam_0)$. 
Since the fundamental group of 
the base curve does not make any effect on
this multiplication, the exact same formula
for $H_{g,n}^B(\vec{\mu})$ holds, as for
the simple Hurwitz numbers for $\bP^1$.
The only difference is
the number $r$ of 
simple ramification points. 

\begin{prop}[The cut-and-join equation]
For all $g\ge 0$, $n\ge 1$, and $(\mu_1,\dots,\mu_n)
\in \bZ_+^n$ subject to 
\begin{equation}
\label{eq:r>0}
2g-2+n-\big(2g(B)-1\big)\sum_{i=1}^n \mu_i >0,
\end{equation}
the simple Hurwitz numbers
$H_{g,n}^B(\mu_1,\dots,\mu_n)$
of degree $d$ satisfy the 
following equation:
\begin{multline}
\label{eq:CAJ}
\left(2g-2+n-\big(2g(B)-1\big)d\right) 
H_{g,n}^B(\mu_{[n]})
=\half\sum_{i\neq j}
\left(\mu_i+\mu_j\right)
H_{g,n-1}^B\left(\mu_i+\mu_j, 
\mu_{[\hat{i},\hat{j}]}\right)
\\
+\half\sum_{i=1}^n\sum_{\alpha+\beta=\mu_i} 
\alpha\beta 
\left(H_{g-1,n+1}^B(\alpha,\beta,\mu_{[\hat{i}]})+
\sum_{\substack{g_1+g_2=g\\ I\sqcup J=[\hat{i}]}}
H_{g_1,|I|+1}^B(\alpha,\mu_I)\;
H_{g_2,|J|+1}^B(\beta,\mu_J)\right).
\end{multline}
Here and throughout the paper we use the 
following notational convention. $[n]=\{1,\dots,n\}$ is
an index set, and $[\hat{i}] = [n]\setminus \{i\}$, etc.
For a subset $I\subset [n]$, $\mu_I = (\mu_i)_{i\in I}$.
\end{prop}

\begin{rem}
\label{rem:CAJ}
Unlike the case of $\bP^1$, the cut-and-join equation
does not determine all values of 
simple Hurwitz numbers. When the degree $d$ takes
its maximum value $\frac{2g-2+n}{2g(B)-1}$, 
the equation 
gives a trivial equality $0=0$. 
\end{rem}

The genus $1$ simple Hurwitz numbers 
based on an elliptic curve $B=E$ are  
easy to calculate. From (\ref{eq:nd}) we have
$n=d$ and $r=0$ when $g=g(B)=1$. Therefore, the
covering is unramified. If we allow 
disconnected domain, then the total number
is equal to the number of partitions
$$
p(d)=
\big|\Hom\big(\pi_1(E),S_d\big)/\!\!/S_d\big|
$$
of  degree $d$
(see for example, \cite{D}). Let 
$$
\phi(q) = \prod_{m=1}^\infty (1-q^m)
$$
be the Euler function. Then we have
$$
\sum_{n=1}^\infty \frac{1}{n!} H_{1,n}^E(1,\dots,1)
q^n = -\log \phi(q) =
\sum_{n=1}^\infty 
\left(\sum_{m|n} \frac{1}{m}\right) q^n,
$$
hence
\begin{equation}
\label{eq:H1n}
H_{1,n}^E(1,\dots,1) = n! \sum_{m|n} \frac{1}{m} = (n-1)! \sigma(n), 
\end{equation}
where $\sigma$ is the sum of  divisors function. 
This is an integer sequence, 
and its first ten terms are
 $1, 3, 8, 42, 144, 1440, 5760, 75600, 
 524160, 6531840$. This 
 sequence has many interesting properties
 (see OEIS, A038048). 

More generally, let us consider 
the case when the equality holds in 
(\ref{eq:g and g(B)}). We need this analysis
 in Section~\ref{sect:Sch}. 
 Because of (\ref{eq:nd}), $g-1=n\big(g(B)-1\big)$
 implies $d=n$ and $r=0$, hence the covering
 $f:C\lrar B$ that is counted 
 is totally unramified. Here again 
 the number of disconnected unramified
 coverings is given by the classical dimension 
 formula (see for example \cite{MY} for 
 elementary derivations):
 $$\left|\Hom\big(\pi_1(B),S_d\big)/\!\!/S_d\right|
 = \sum_{\lam\vdash d}
 \left( \frac{\dim \lam}{d!}\right)^{\rchi(B)}.
 $$
Here $\lam\vdash d$ parametrizes irreducible
representations of $S_d$, and $\dim\lam$ is its
dimension.

\section{The discrete Laplace transform}
\label{sect:LT}

The discrete Laplace transform 
$F_{g,n}^B(x_1,\dots,x_n)$ of (\ref{eq:Fgn}) is a polynomial
 of degree $\frac{2g-2+n}{2g(B)-1}$
  with the lowest degree
term $H_{g,n}^B(1,1,\dots,1) x_1\cdots x_n$. 
The following proposition 
is an analogue of the case of 
simple Hurwitz numbers of $\bP^1$,
and the proof is exactly the same as that of
\cite[Lemma~4.1]{BHLM}.

\begin{prop}
\label{prop:LaplaceTf}
The discrete Laplace transform of the
cut-and-join equation {\rm{(\ref{eq:CAJ})}} 
is precisely the 
 partial differential  equation 
{\rm{(\ref{eq:PDE})}}.
\end{prop}

Reflecting Remark~\ref{rem:CAJ},
the differential equation (\ref{eq:PDE}) alone
does not
determine free energies
$F_{g,n}^B(x_1,\dots,x_n)$. This is because
every homogeneous function of degree
$\frac{2g-2+n}{2g(B)-1}$
 is in the kernel of the Euler
differential operator 
$$
2g-2+n-
\big(2g(B)-1\big)
\sum_{i=1}^n x_i\frac{\partial}{\partial x_i}
$$
on the left-hand side of (\ref{eq:PDE}).
Thus the highest degree terms of 
free energies are not determined by this
differential equation.
We will determine the 
homogeneous highest degree 
terms of the free energies 
in a different method in Section~\ref{sect:heat}.

The genus $1$ elliptic Hurwitz numbers
(\ref{eq:H1n}) yields 
\begin{equation}
\label{eq:F1n}
F_{1,n}^E(x_1,\dots,x_n) = \left(
n!\sum_{m|n} \frac{1}{m}\right)
x_1\cdots x_n.
\end{equation}

\section{A Schr\"odinger equation}
\label{sect:Sch}

In this section we prove (\ref{eq:Q}), 
provided that $g(B)> 0$. We remark here 
that for $g(B) = 0$, the proof is rather 
different \cite{MS}, yet the same formula
holds.

\begin{thm}
  The partition function {\rm{(\ref{eq:ZB})}}
  satisfies the Schr\"odinger-type equation
  {\rm{(\ref{eq:Q})}}:
\begin{displaymath}
\left[
\frac{\partial}{\partial \hbar}-\half
\left(x \frac{\partial}{\partial x}\right)^2
+\left(\half-\frac{2g(B)-1}{\hbar}\right)
x \frac{\partial}{\partial x}
\right] Z^B(x,\hbar)
=0.
\end{displaymath}
\end{thm}

\begin{proof} We use the same method
of the proof of \cite[Theorem~5.1, Theorem~5.3, 
Appendix~A]{MS}.
When (\ref{eq:nd}) holds, 
 the diagonal evaluation of 
  (\ref{eq:PDE}) yields
  \begin{equation}
    \label{eq:diag-PDE}
    \begin{split}
      &\left(2g-2+n-\big(2g(B)-1\big)
      x\frac{d}{dx}\right) 
      F_{g,n}^B(x,\dots,x)\\
      &=\frac{n(n-1)}{2}x^2
      \left.\frac{\partial^2}{\partial u^2}\right|_{u=x}
      F_{g,n-1}^B(u,x,\dots,x)\\
      &+\frac{n}{2}x^2
      \left.\frac{\partial^2}{\partial u_1\partial u_2}\right|_{u_1=u_2=x}
      F_{g-1,n+1}^B(u_1,u_2,x,\dots,x)\\
      &+\frac{n!}{2}\sum_{\substack{g_1+g_2=g\\n_1+n_2=n-1}}
      \frac{x^2}{(n_1+1)!(n_2+1)!}
      \frac{d}{dx}F_{g_1,n_1+1}^B(x,\dots,x)\cdot
      \frac{d}{dx}F_{g_2,n_2+1}^B(x,\dots,x).
    \end{split}
  \end{equation}
  For $m\ge 2 g(B)$, let us  define 
  \begin{equation*}
    S_m(x):=\sum_{2g-2+n=m-1}
    \frac{1}{n!}F_{g,n}^B(x,\dots,x).
  \end{equation*}
  Note that $S_m$ contains a contribution from 
  domain curves with 
   $g-1= n(g(B)-1)$, or equivalently,
    $d=n$ 
  and $r=0$, 
  for which the cut-and-join equation
  is not valid. 
  Therefore,  when we deduce an equation 
  for $S_m$'s, we need to remove these 
  boundary terms
  from the equation. Now from 
   (\ref{eq:diag-PDE}) we obtain
     \begin{equation}
    \begin{aligned}
    \label{eq:adjustment}
      &\left(m-\big(2g(B)-1\big)x\frac{d}{dx}\right)
      S_{m+1}(x) 
      \\
      &
      -\half x^2\frac{d^2}{d x^2} 
      S_m(x)
      -\half\sum_{m_1+m_2=m+1}
      x\frac{d}{dx}S_{m_1}(x)\cdot
      x\frac{d}{dx}S_{m_2}(x)\\
      =&
      \sum_{\substack{2g-2+n=m\\g-1= n(g(B)-1)}}
      \frac{1}{n!}
      \left(m-\big(2g(B)-1\big)x\frac{d}{dx}\right)
      F_{g,n}^B(x,\dots,x)\\
      &-\half
      \sum_{\substack{2g-2+n=m-1\\g-1=(n+1)(g(B)-1)}}
      \frac{1}{(n-1)!}
      x^2
      \left.
      \frac{\partial^2}{\partial u^2}
      \right|_{u=x}
      F_{g,n}^B(u,x,\dots,x)\\
      &-\half
      \sum_{\substack{2g-2+n=m-1\\g=(n-1)(g(B)-1)}}
      \frac{1}{(n-2)!}
      x^2
      \left.
      \frac{\partial^2}{\partial u_1\partial u_2}
      \right|_{u_1=u_2=x}
      F_{g,n}^B(u_1,u_2,x,\dots,x)\\
      &-\half
      \sum_{\substack{2g-2+n=m\\g-1= n(g(B)-1)}}
      \sum_{\substack{g_1+g_2=g\\n_1+n_2=n+1}}
    \frac{x^2}{n_1!\,n_2!}
        \left(  \frac{d}{dx}
F_{g_1,n_1}^B(x,\dots,x)\right)
    \left( \frac{d}{dx}
      F_{g_2,n_2}^B(x,\dots,x)\right).
    \end{aligned}
  \end{equation}
  We note here that in the boundary 
  contribution $g-1=n\big(g(B)-1\big)$, 
  since $n=d$, the free energies are 
  single monomials proportional to 
  $x_1x_2\cdots x_n$.
  Let us look at the right-hand side
  of (\ref{eq:adjustment}).
  The first line of the
  right-hand side is $0$ because 
 $F_{g,n}^B(x,\dots,x) = H_{g,n}^B(1,\dots,1)x^n$,
 and $m=2g-2+n=n\big(2g(B)-1\big)$, since $r=0$.
The second line is also $0$, because 
$F_{g,n}^B(u,x,\dots,x)$ is linear in $u$.

  The third summation on the right-hand
  side is empty, because
  we need 
  $n\big(g(B)-1\big)\le g-1 =(n-1)\big(g(B)-1\big)
  -1$,
  which does not happen. 
  Similarly, the fourth line summation is also 
  empty, because we are requiring
  $g=g_1+g_2,n+1=n_1+n_2$, and
  \begin{align*}
  n_1\big(g(B)-1\big) &\le g_1-1\\
  n_2\big(g(B)-1\big) &\le g_2-1\\
  n\big(g(B)-1\big) &= g-1.
  \end{align*}

  We have thus obtained a
  recursion equation 
    \begin{multline}
    \label{eq:S-recur}
    \left(m-(2g(B)-1)x\frac{d}{dx}\right)
    S_{m+1}(x) 
    \\
    =\half x^2\frac{d^2}{d x^2} 
    S_m(x)
    +\half\sum_{m_1+m_2=m+1}
    x\frac{d}{dx}S_{m_1}(x)\cdot
    x\frac{d}{dx}S_{m_2}(x).
  \end{multline}
  In terms of the generating function 
\begin{displaymath}
  F(x,\hbar) = \sum_{m=1}^\infty\hbar^{m-1}S_m(x),
\end{displaymath}
 (\ref{eq:S-recur})
becomes 
\begin{equation}
  \label{eq:F-PDE}
  \left(\frac{\partial}{\partial\hbar}-
    \frac{2g(B)-1}{\hbar}x\frac{d}{dx}\right)
  F(x,\hbar)=
  \half x^2\frac{d^2}{d x^2}F(x,\hbar)+
  \half \left(x\frac{d}{dx}F(x,\hbar)\right)^2.
\end{equation}
Since 
$Z^B(x,\hbar)=\exp F(x,\hbar)$,  
(\ref{eq:Q})  follows directly from 
(\ref{eq:F-PDE}).
This completes the proof.
\end{proof}

Using the Schr\"odinger equation (\ref{eq:Q}),
we can determine the form of the solution 
$Z^B(x,\hbar)$.

\begin{lem}
The partition function has the following 
expansion:
\begin{equation}
\label{eq:Z-expansion}
Z^B(x,\hbar) = \sum_{m=0} ^\infty
c_m e^{\half m(m-1)\hbar} 
\left(\hbar^{1-\rchi(B)} x\right)^m,
\end{equation}
where $c_m$ is a constant.
\end{lem}

\begin{proof}
The partition function $Z^B(x,\hbar)$ 
of (\ref{eq:ZB}) is a 
formal power series in $x$ and $\hbar$. 
Thus it has an expansion of the form
$$
Z^B(x,\hbar) = \sum_{m=0} ^\infty
c_m f_m(\hbar) x^m
$$
with $f_m(\hbar)\in \bC[[\hbar]]$.
Then 
the Schr\"odinger equation
(\ref{eq:Q}) yields an ordinary
differential equation 
$$
f_m ' = \left(\half m(m-1) +m
\frac{1-\rchi(B)}{\hbar}\right)f_m,
$$
whose solution is 
$$
f_m(\hbar) = c\; e^{\half m(m-1)\hbar} 
\hbar ^{m\big(1-\rchi(B)\big)}.
$$
This completes the proof.
\end{proof}

\section{The heat equation and its consequences}
\label{sect:heat}

As we have seen in Section~\ref{sect:Sch},
the Schr\"odinger equation (\ref{eq:Q})
determines the solution 
$Z^B(x,\hbar)$ only up to the
form (\ref{eq:Z-expansion}).
To determine the coefficients, we need
another technique.

In this section we use a heat equation
technique of \cite{Kazarian}. First we remark that the
cut-and-join equation gives rise to a
heat equation for another generating
function of base $B$ Hurwitz numbers. 
To determine a solution of a heat
equation, we need to identify
the initial value. 
We will show that the initial condition 
exactly corresponds to determining the
highest degree terms of $F_{g,n}^B(x_1,\dots,x_n)$.
The exponential generating function 
of these highest degree terms 
can be determined by a character formula
of \cite{OP2}. Thus we obtain
the unique solution of the heat equation,
which in turn gives all base $B$ Hurwitz numbers.

The generating function for
base $B$ Hurwitz numbers we consider is
\begin{align}
\label{eq:H}
\mathbf{H}(t,\bp) &= \sum_{g=0}^\infty
\sum_{n=1} ^\infty \mathbf{H}_{g,n}(t,\bp),
\\
\label{eq:Hgn}
\mathbf{H}_{g,n}(t,\bp) &= 
\frac{1}{n!} \sum_{\vec{\mu}\in \bZ_+^n}
H_{g,n}^B(\vec{\mu}) \bp_\mu t^{r(g,\mu)},
\end{align}
where $r(g,\mu) = 2g-2+n
 -|\mu|\big(1-\rchi(B)\big)$ is the 
number of simple ramification points,
and $\bp_\mu = p_{\mu_1}\cdots p_{\mu_n}$.
The same argument of \cite{Goulden,Kazarian,
KazarianLando,O,Zhou1} shows that
$e^{\mathbf{H}(t,\bp)}$ satisfies
a \emph{heat equation}, that is obtained from
the cut-and-join equation.
 For a partition $\mu = (\mu_1\ge \mu_2\ge\cdots)$ 
 of a finite length $\ell(\mu)$,
we define the \textbf{shifted power-sum function}
by
\begin{equation}
\label{eq:shifted power-sum}
\bp_r[\mu] := \sum_{i=1}^\infty
\left[
\left(\mu_i-i+\half\right)^r-
\left(-i+\half\right)^r
\right].
\end{equation}
This is a finite sum of $\ell(\mu)$ terms.
Then we have \cite{Goulden, Zhou1}
\begin{equation}
\label{eq:eugenfunction}
\sum_{i,j\ge 1} \left(
(i+j)p_ip_j\frac{\partial}{\partial p_{i+j}}
+ijp_{i+j}\frac{\partial^2}{\partial p_i \partial p_j}
\right)
s_\mu(\bp)
= \bp_2[\mu] 
\cdot 
s_\mu(\bp),
\end{equation}
where $s_\mu(\bp)$ is the Schur function 
defined by
\begin{equation}
\label{eq:Schur}
s_\mu(\bp) = \sum_{|\lam|=|\mu|}
\frac{\rchi_\mu(\lam)}{z_\lam} \bp_\lam,
\quad
z_\mu = \prod_{i=1}^{\ell(\mu)} m_i! i^{m_i},
\end{equation}
$m_i =$  the number of parts in $\mu$ of 
length $i$,
and $\rchi_\mu(\lam)$ is  the value of the irreducible 
character of the representation $\mu$ 
of the symmetric group evaluated at 
the conjugacy class $\lam$.
Let us denote the \emph{cut-and-join} operator by
$$
\triangle =\half \sum_{i,j\ge 1} \left(
(i+j)p_ip_j\frac{\partial}{\partial p_{i+j}}
+ijp_{i+j}\frac{\partial^2}{\partial p_i \partial p_j}
\right).
$$
Then Schur functions are eigenfunctions of the 
operator:
$$
\triangle s_\mu(\bp) = \half \bp_2[\mu]
\cdot s_\mu(\bp),
$$
and the
cut-and-join equation of simple Hurwitz numbers
(\ref{eq:CAJ})
yields a heat equation 
\begin{equation}
\label{eq:heat}
\frac{\partial}{\partial t} e^{\mathbf{H}(t,\bp)}
=\triangle e^{\mathbf{H}(t,\bp)}.
\end{equation}
We can solve a
heat equation by the method of eigenfunction 
expansion. Thus we have an expression
\begin{equation}
\label{eq:Schur expansion}
e^{\mathbf{H}(t,\bp)}
=\sum_\mu a_\mu s_\mu(\bp)e^{\half \bp_2[\mu] t}
\end{equation}
for a constant $a_\mu$ associated with every 
partition $\mu$. These constants are determined by
the initial value $t=0$. From (\ref{eq:Hgn})
we see that
the initial value comes from the cases 
when $r(g,\mu)=0$.

First we note that (\ref{eq:r}) implies that
 simple Hurwitz numbers 
with $r=0$ correspond to the case with only 
$1$ branched point on the base curve $B$. 
If we allow disconnected domain curves to 
cover $B$, then the number of coverings
of degree $d$
with a prescribed ramification data given 
by a partition $\mu \vdash d$ over $0\in B$
with no other ramification points is easy to 
calculate. Let us denote by 
$H_{d}^{B\bullet}( \mu)$ 
such a number, where we do not label the
inverse images of $0\in B$ this time. 
Then from \cite[Eq.(0.10)]{OP2} we obtain
\begin{equation}
\label{eq:r=0}
H_{d}^{B\bullet}( \mu)
= \sum_{\lam \vdash d}
\left(\frac{\dim \lam}{d!}\right) ^{\rchi(B)}
|C_\mu|\;
\frac{\rchi_\lam(\mu)}{\dim \lam}.
\end{equation}
Here $C_\mu$ is the conjugacy class
of the permutation group $S_d$
determined by the cycle type $\mu$, 
a partition $\lam\vdash d$ is a label
of an irreducible representation of
$S_d$, and $\dim \lam$ is its dimension.
The cardinality of the conjugacy class is
 given by
\begin{equation}
\label{eq:conj}
|C_\mu| = \frac{d!}{\prod_{i} m_i! i^{m_i}}
=\frac{d!}{z_\mu}.
\end{equation}

The generating function of these disconnected
simple Hurwitz numbers can be calculated,
appealing to (\ref{eq:Schur}) and (\ref{eq:conj}),
 as follows:
 \begin{equation}
 \begin{aligned}
\label{eq:r=0,disconn}
\sum_{d=1}^\infty \sum_{\mu\vdash d}
H_d^{B\bullet}(\mu) \bp_\mu
&=
\sum_{d=1}^\infty \sum_{\mu\vdash d}
\sum_{\lam \vdash d}
\left(\frac{\dim \lam}{d!}\right) ^{\rchi(B)}
 |C_\mu|\;
\frac{\rchi_\lam(\mu)}{\dim \lam}\bp_\mu
\\
&=
\sum_{d=1}^\infty 
\sum_{\lam \vdash d}
\frac{d!}{\dim \lam}
\left(\frac{\dim \lam}{d!}\right) ^{\rchi(B)}
\sum_{\mu\vdash d}
\frac{\rchi_\lam(\mu)}{z_\mu}
\bp_\mu
\\
&=
\sum_{\lam}
\left(\frac{|\lam|!}{\dim \lam}\right)^{1-\rchi(B)}
s_{\lam}(\bp),
\end{aligned}
\end{equation}
where the last sum runs over all partitions $\lam$.

Note that $e^{\mathbf{H}(t,\bp)}$
is the generating function of 
simple Hurwitz numbers allowing
disconnected domain curves. 
Therefore, the initial value
$e^{\mathbf{H}(0,\bp)}$ counts the
disconnected base $B$ Hurwitz numbers
with only one branched point at $0 \in B$. In 
other words, we have determined the 
initial condition by (\ref{eq:r=0,disconn}).
The result is
\begin{equation}
\label{eq:initial}
e^{\mathbf{H}(0,\bp)}
=\sum_{\mu}
\left(\frac{|\mu|!}{\dim \mu}\right)^{1-\rchi(B)}
s_{\mu}(\bp)
=
\sum_\mu a_\mu s_\mu(\bp).
\end{equation}
Since Schur functions are linear basis for
symmetric functions, we establish the following.

\begin{thm}
\label{thm:expH(t,p)}
The exponential generating function of the base $B$
simple 
Hurwitz numbers is given by
\begin{equation}
\label{eq:expH(t,p)}
e^{\mathbf{H}(t,\bp)} =  
 \sum_{\mu}
\left(\frac{|\mu|!}{\dim \mu}\right)^{1-\rchi(B)}
s_{\mu}(\bp) e^{\half \bp_2[\mu] t}.
\end{equation}
\end{thm}

\section{The quantum curve}
\label{sect:quantum}

As noted in \cite{MS}, the diagonal
specialization $F_{g,n}^B(x,x,\dots,x)$
corresponds to substituting the power-sum 
symmetric function 
\begin{equation}
\label{eq:pj}
p_j = x_1 ^j+x_2 ^j+x_3 ^j + \cdots
\end{equation}
by its principal
specialization $p_j=x^j$.
More precisely, we have the following.

\begin{lem}
\label{lem:PS}
Let us define the  
\textbf{principal specialization} 
of the power-sum symmetric functions by
\begin{align*}
p_j(s) &= \left(s^{1-\rchi(B)} x\right)^j
\\
\bp_\mu (s) &=p_{\mu_1}(s)\cdots p_{\mu_n}(s)
=\left(s^{1-\rchi(B)} x\right)^{|\mu|}.
\end{align*}
Then we have
\begin{equation}
\label{eq:Z=eH}
Z^B(x,\hbar) = e^{\mathbf{H}(\hbar,\bp(\hbar))}.
\end{equation}
\end{lem}

\begin{proof}
We have 
\begin{multline*}
\mathbf{H}(\hbar,\bp(\hbar))
=
\sum_{g,n\ge 1} \sum_{\mu\in\bZ_+^n}\frac{1}{n!} 
\hbar^{2g-2+n-|\mu|\big(1-\rchi(B)\big)}
H_{g,n}^B (\mu_1,\dots,\mu_n) \bp_\mu(\hbar)
\\
= 
\sum_{g,n\ge 1} \frac{1}{n!} 
\hbar^{2g-2+n}
F_{g,n}^E (x,x,\dots,x) ,
\end{multline*}
which yields (\ref{eq:Z=eH}).
\end{proof}

From the expansion formulas (\ref{eq:expH(t,p)})
and (\ref{eq:Z-expansion}), together with the
equality (\ref{eq:Z=eH}), we obtain 
\begin{equation}
\label{eq:PS}
\sum_{\mu} 
\left(\frac{|\mu|!}{\dim \mu}\right)^{1-\rchi(B)} 
e^{\half \bp_2[\mu] \hbar}s_\mu\big(\bp(\hbar)\big)
=
\sum_{m=0}^n c_m e^{\half m(m-1)\hbar}
\left(\hbar^{1-\rchi(B)} x\right)^m.
\end{equation}
As explained in \cite[Section~6]{MS}
and also in \cite{MSS},
this equality is exactly the effect of the principal
specialization of Lemma~\ref{lem:PS},
which turns the summation over
all partitions into the sums over just
one-row partitions.
We have thus determined the coefficients
$c_m$ in the expansion (\ref{eq:Z-expansion}).
It is given by 
\begin{equation}
\label{eq:cm}
c_m = (m!)^{1-\rchi(B)}.
\end{equation}
This completes the proof of (\ref{eq:ZB sum}).
The convergence of the infinite series is obvious 
from the shape of 
(\ref{eq:ZB sum}) or (\ref{eq:ZB in tau}), 
since 
$$
\limsup_{m\rar \infty}
 \left| (m!)^{1-\rchi(B)} e^{\half m(m-1)
\hbar}\right|^{\frac{1}{m}} = 0
$$
if $Re(\hbar)<0$.

The expansion (\ref{eq:ZB sum}) allows us to 
derive the quantum curve-type equation
(\ref{eq:P}). Since
\begin{multline*}
\left[\hbar^{1-\rchi(B)} x e^{\hbar x\frac{d}{dx}}
\left( \frac{d}{dx} x\right)^{1-\rchi(B)}
\right]
(m!)^{1-\rchi(B)}e^{\half m(m-1)\hbar}
\left(\hbar^{1-\rchi(B)} x\right)^m
\\
=
\big((m+1)!\big)^{1-\rchi(B)}e^{\half m(m+1)\hbar}
\left(\hbar^{1-\rchi(B)} x\right)^{m+1}
\end{multline*}
for every $m\ge 0$,
we have 
\begin{equation}
\label{eq:P-1}
\left(1 -\hbar^{1-\rchi(B)} x e^{\hbar x\frac{d}{dx}}
\left( \frac{d}{dx} x\right)^{1-\rchi(B)}
\right) Z^B(x,\hbar) = 1.
\end{equation}
By differentiating (\ref{eq:P-1}) we obtain
(\ref{eq:P}).

\begin{lem} Let
\begin{align}
\label{eq:opP}
P&=\hbar x\frac{d}{dx}
\left(1 -\hbar^{1-\rchi(B)} x e^{\hbar x\frac{d}{dx}}
\left( \frac{d}{dx} x\right)^{1-\rchi(B)}
\right) ,
\\
Q&=\frac{\partial}{\partial \hbar}
-\half \left(x\frac{\partial}{\partial x}\right)^2
+\left(\half -\frac{1-\rchi(B)}{\hbar}\right)
x\frac{\partial}{\partial x}.
\label{eq:opQ}
\end{align}
 Then
 \begin{equation*}
 [P,Q]=-\frac{1}{\hbar}P.
 \end{equation*}
\end{lem}
\begin{proof}
  We first note that $Q$ commutes with 
  $x\frac{d}{dx}$
  and $\frac{d}{dx}x$. 
  Since
  \begin{align*}
&  \left[\hbar^{1-\rchi(B)} x e^{\hbar x\frac{d}{dx}},
  Q\right]
  \\
  &= \left[
  \hbar^{1-\rchi(B)} x e^{\hbar x\frac{d}{dx}},
  \frac{\partial}{\partial \hbar}
  \right]
  +\left[
  \hbar^{1-\rchi(B)} x, 
 -\half \left(x\frac{\partial}{\partial x}\right)^2
+\left(\half -\frac{1-\rchi(B)}{\hbar}\right)
x\frac{\partial}{\partial x}
  \right] e^{\hbar x\frac{d}{dx}}
  \\
  &=
  -\left(1-\rchi(B)\right) 
  \hbar^{-\rchi(B)}xe^{\hbar x\frac{d}{dx}}
  -\hbar^{1-\rchi(B)} x^2\frac{d}{dx}
  e^{\hbar x\frac{d}{dx}}
  \\
  &\qquad+ \hbar^{1-\rchi(B)} x^2\frac{d}{dx}
  e^{\hbar x\frac{d}{dx}}
  +
\half  \hbar^{1-\rchi(B)} x
  e^{\hbar x\frac{d}{dx}}
  -\hbar^{1-\rchi(B)} 
  \left(\half -\frac{1-\rchi(B)}{\hbar}\right) 
xe^{\hbar x\frac{d}{dx}}
\\
&=0,
  \end{align*}
  (\ref{eq:[P,Q]}) follows from
  $$
  0 = \left[\frac{1}{\hbar}P,Q\right] 
  = \frac{1}{\hbar}[P,Q]
  +\left[\frac{1}{\hbar},Q\right]P
  =\frac{1}{\hbar}[P,Q]
  +\frac{1}{\hbar^2}P.
  $$
\end{proof}

\section{Semi-classical limit}
\label{sect:semi}

The semi-classical analysis of the operators
$P$ and $Q$ are performed in the following 
way. Suppose our counting problem had 
genus $0$ contributions $F_{0,n}^B(x_1,\dots,x_n)$.
Then we define
\begin{equation}
\label{eq:Sm}
S_m(x) :=\sum_{2g-2+n=m-1} \frac{1}{n!}
F_{g,n}^B(x,\dots,x)
\end{equation} 
as before, and consider a formal expression
\begin{equation}
\label{eq:Z-extra}
\overline{Z}^B(x,\hbar) = e^{\sum_{m=0}^\infty
\hbar^{m-1} S_m(x)}
= e^{\frac{1}{\hbar} S_0(x)
+ S_1(x)} Z^B(x,\hbar).
\end{equation}
Let us introduce a variable $u$ such that 
 $x=e^u$, and regard  the coefficients $S_m$ as
 functions in $u$. Since 
  $$
 x\frac{d}{dx} = \frac{d}{du},
 $$
we have
 $$
 P = \hbar \frac{d}{du}
 \left(1 -\hbar^{1-\rchi(B)}
  e^u e^{\hbar \frac{d}{du}}\left(
 1+\frac{d}{du}\right)^{1-\rchi(B)}\right) .
 $$
Then 
\begin{equation}
\begin{aligned}
\label{eq:semi-classicalP}
e^{-S_1}e^{-\frac{1}{\hbar} S_0}
P e^{\frac{1}{\hbar} S_0}e^{S_1}
&=S_0' -e^u \big(S_0'\big)^{2-\rchi(B)}
 e^{\frac{1}{\hbar}
\big(S_0(u+\hbar)-S_0(u)\big)} +O(\hbar)
\\
&=
S_0' -e^u \big(S_0'\big)^{2-\rchi(B)}
 e^{S_0'} + O(\hbar),
\end{aligned}
\end{equation}
where $'=\frac{d}{du} =  x\frac{d}{dx}$, and
by $O(\hbar)$ we mean an operator
whose application to the partition function
$Z^B(x,\hbar)$ produces a function of 
order $1$ or higher in $\hbar$.
Now define a new variable 
\begin{equation}
\label{eq:y}
y = S_0'.
\end{equation}
Then the semi-classical limit $\hbar\rar 0$
of (\ref{eq:semi-classicalP})
yields an equation
\begin{equation}
\label{eq:y Lambert}
y-xy^{2-\rchi(B)}e^y = y\left(
1-xy^{1-\rchi(B)}e^y 
\right)=0,
\end{equation}
since $Z^B(x,0) = 1$. Note that this is exactly
the
\emph{total symbol} of the operator 
$P$, where $\hbar x\frac{d}{dx}$ is 
represented by a commutative variable $y$.
The second factor of (\ref{eq:y Lambert}) 
also recovers 
the Lambert curve
(\ref{eq:B-Lambert}).

Although the formal manipulation 
seems to work,
however,
 we have to remember that we have
derived the operator $P$ assuming the
shape of the solution $Z^B(x,\hbar)$
as in (\ref{eq:ZB sum}). Since we
are imposing
$$
P \overline{Z}^B(x,\hbar) =0,
$$ 
it forces that $y=0$, which makes
(\ref{eq:y Lambert}) trivially correct.

On the other hand, 
since the kernel  of 
$Q$ assums only the
expansion (\ref{eq:Z-expansion}), where
the summation index $m$ can be negative,
the semi-classical analysis of  $Q$  
 does go through.

\begin{equation}
\begin{aligned}
\label{eq:semi-classicalQ}
&e^{-S_1}e^{-\frac{1}{\hbar} S_0}
Q e^{\frac{1}{\hbar} S_0}e^{S_1}
\\
&=
e^{-S_1}e^{-\frac{1}{\hbar} S_0}
\left[
\frac{\partial}{\partial \hbar}
-\half \frac{\partial^2}{\partial u^2}
+\left(\half-\frac{1-\rchi(B)}{\hbar}\right)\frac{\partial}
{\partial u}
\right]
 e^{\frac{1}{\hbar} S_0}e^{S_1}
\\
&=
-\frac{1}{\hbar^2} \left(
S_0+\half \big(S_0'\big)^2 +
 \big(1-\rchi(B)\big)S_0'\right)
 \\
 &\qquad
-\frac{1}{\hbar}
\left(\half S_0''-\half S_0'+S_0'S_1'+
\big(1-\rchi(B)\big)S_1'
\right)+O(1).
\end{aligned}
\end{equation}
For the $\hbar\rar 0$ limit to exist, we need
\begin{align}
\label{eq:1}
&S_0+\half \big(S_0'\big)^2 + 
 \big(1-\rchi(B)\big)S_0'=0,
\\
\label{eq:2}
&\half S_0''-\half S_0'+S_0'S_1'+
\big(1-\rchi(B)\big)S_1'=0.
\end{align}
From (\ref{eq:1}) we obtain 
\begin{equation}
\label{eq:S0 in y}
S_0 = -\half y^2 - \big(1-\rchi(B)\big)y,
\end{equation}
or equivalently,
\begin{equation}
\label{eq:dydu}
y=-\big(y+1-\rchi(B)\big)\frac{dy}{du}.
\end{equation}
Its solution is 
$$
y+\log y^{1-\rchi(B)} = -u +\const.
$$
If we take the constant of integration to be $0$, then
we obtain
$$
e^{u} = x =y^{\rchi(B)-1}e^{-y},
$$
recovering (\ref{eq:B-Lambert}). 
From (\ref{eq:2})
we have 
$$
\half y'-\half y+\big(y+1-\rchi(B)\big) S_1'=0,
$$
or equivalently
$$
\frac{ dS_1}{du} =-\half \frac{1}{y+1-\rchi(B)}
\frac{dy}{du}
+\half \frac{y}{y+1-\rchi(B)}\frac{dy}{dy}.
$$
Since we know $du/dy$ from (\ref{eq:dydu}),
we can integrate the above equation to obtain
\begin{equation}
\label{eq:S1 in y}
S_1 = -\half  y+\half \log \big(y+1-\rchi(B)\big) 
+\const.
\end{equation}

The above derivation of the semi-classical
limit of the operator $Q$ is valid if
$Z^B(x,\hbar)$ contains negative powers of
$\hbar$. If we assume the expansion
(\ref{eq:ZB sum}), then such a situation 
occurs when the base curve $B$ satisfies
$\rchi(B)>1$. Indeed, our formulas
(\ref{eq:S0 in y}) and (\ref{eq:S1 in y}) 
agree with those of 
\cite[Section~5]{MS} for $B=\bP^1$ 
with $y=z$, and
\cite[Section~6]{BHLM} for
$B=\bP^1[a]$ with $y=z^a$.

\begin{ack}
  X.L.\ received the  
 China Scholarship Council 
 grant CSC-2010811063, which allowed him
 to conduct research at the 
 Department of Mathematics, University of 
 California, Davis. 
 He is also supported by the National Science 
 Foundation of China grants No.11201477, 11171175, 
 and the
 Chinese Universities Scientific Fund No.2011JS041.
The research of M.M.\ is 
supported by NSF grant
DMS-1104734.
A.S.\ received support from
the U.S.\ Government.
\end{ack}


\providecommand{\bysame}{\leavevmode\hbox to3em{\hrulefill}\thinspace}

\bibliographystyle{amsplain}

\end{document}